\documentclass[10pt]{article}
\usepackage{BAstyle}
\usepackage[margin=1.5in]{geometry}
\usepackage{mathrsfs}

\newtheorem*{notation*}{Notation}
\newtheorem*{LLN*}{Asymptotic Law of Large Numbers}
\newtheorem{heuristic}[theorem]{Heuristic}

\title{Counting Functions for Random Objects in a Category}
\author{Brandon Alberts}
\date{}

\begin{document}
\maketitle

\abstract{
In arithmetic statistics and analytic number theory, the asymptotic growth rate of counting functions giving the number of objects with order below $X$ is studied as $X\to \infty$. We define general counting functions which count epimorphisms out of an object on a category under some ordering. Given a probability measure $\mu$ on the isomorphism classes of the category with sufficient respect for a product structure, we prove a version of the Law of Large Numbers to give the asymptotic growth rate as $X$ tends towards $\infty$ of such functions with probability $1$ in terms of the finite moments of $\mu$ and the ordering. Such counting functions are motivated by work in arithmetic statistics, including number field counting as in Malle's conjecture and point counting as in the Batyrev-Manin conjecture. Recent work of Sawin--Wood gives sufficient conditions to construct such a measure $\mu$ from a well-behaved sequence of finite moments in very broad contexts, and we prove our results in this broad context with the added assumption that a product structure in the category is respected. These results allow us to formalize vast heuristic predictions about counting functions in general settings.% Random groups have previously been used the give heuristic models for the distributions of unramified objects (which is a stated motivation for the work of Sawin--Wood), and the results of this paper begin bridging the theoretical gap between heuristics for number field counting and class group statistics by constructing and proving asymptotic growth rates for counting functions on random objects.
}

%\tableofcontents

\newpage

\section{Introduction}

Distributions of arithmetic objects are commonly studied via counting functions. A classic example is the prime number theorem, which determines the asymptotic growth rate of
\[
\pi(X) = \#\{p\text{ prime }\mid p\le X\} \sim \frac{X}{\log X}
\]
as $X$ tends to infinity. Counting functions on subsets of prime numbers is one way that we work to understand the distribution of the prime numbers. Cram\'er made rigorous a heuristic argument of Gauss for predicting the distribution of prime numbers in \cite{cramer1994}, whereby a sequence of independent variables $X_n$ indexed by natural numbers and valued in $\{0,1\}$ are considered, for which $X_n=1$ with probability $\frac{1}{\log(n)}$ as suggested by the prime number theorem. By asking the same distribution questions of the sequence $X_n$ instead of the sequence of prime numbers, Cram\'er is able to prove that this sequence satisfies a number of properties conjectured to be true for the primes with probability $1$. See \cite{granville1995} for a summary of Cram\'er's random model and some improvements.

In arithmetic statistics, other objects are frequently studied in this way. Given an ordering of the objects satisfying a Northcott property, one can ask for the asymptotic growth rate of the counting function
\begin{align}\label{eq:counting}
\#\{\text{object} \mid \text{order(object)}\le X\}
\end{align}
as $X$ tends to $\infty$. The standard examples are points on schemes (or stacks) bounded by height, whose asymptotic growth rates are predicted by the Batyrev-Manin conjecture \cite{franke-manin-tschinkel1989,batyrev-manin1990}, and counting $G$-extensions of a global field $K$ ordered by discriminant, whose asymptotic growth rates are predicted by Malle's conjecture \cite{malle2002,malle2004}.

Malle's conjecture in particular is closely related to the study of class group statistics, where one asks for the distribution of the $p$-parts of the class group of an extension $K/\Q$ as $K$ varies over some family of extensions ordered by discriminant. This distribution is very often predicted to agree with a true probability distribution on random groups. For a non-exhaustive list of examples for statistics of very general unramified objects over a family of global fields see \cite{cohen-lenstra1984,friedman-washington1989,boston-bush-hajir2017,liu-wood-zureick-brown2019}. Despite this close relation, there has been no attempt to make predictions for Malle's counting function using similar random structures. Like with Cram\'er's model, it is reasonable to consider that there exists a random group with some discriminant structure which models the absolute Galois group of the global field $K$ for which Malle's conjecture holds with probability $1$. Potentially the same is true for other counting functions in arithmetic statistics.

Another eye-catching feature of Malle's conjecture is the Galois correspondence. $G$-extensions $L/K$ are in a $1$-to-$|\Aut(G)|$ correspondence with surjective homomorphisms $\Gal(\overline{K}/K)\to G$. The recent work of Sawin--Wood \cite{sawin-wood2022} solved the moment problem for random objects in a category, where they proved in great generality that given a sequence $M_G$ on (isomorphism classes of) the category of finite objects $C$ that ``do not grow too fast'', there exists a measure $\mu$ on (isomorphism classes of) the category of pro-objects $\mathcal{P}$ whose finite moments are given by $M_G$. In the setting of objects in a category, this is taken to mean
\[
\int_{\mathcal{P}}\#{\rm Epi}(\mathscr{G},G)\ d\mu(\mathscr{G}) = M_G
\]
for each finite object $G$. Here we see an immediate similarity with Malle's conjecture - Malle's conjecture can be interpreted as an asymptotic count for the number of epimorphisms from the absolute Galois group of $K$ to a fixed finite group. This suggests that the setting Sawin--Wood work in, which they state was built with statistics of unramified objects in mind, is also an excellent setting to model Malle's conjecture. By extension, random models for other counting functions may also exist in the setting considered by Sawin--Wood as long as they can be interrpretted as counting epimorphisms.

The goal of this paper is to consider counting functions in the broad setting described by Sawin--Wood. The remainder of this introduction is separated into two subsections - the first defining a counting function with respect to an ordering for counting epimorphisms in a category (see Definition \ref{def:countingfunction}), and the second for stating Law of Large Numbers results to determine the asymptotic growth rate of these counting functions with probability $1$ (see Theorem \ref{thm:LLNintro}). We will prove the Law of Large Numbers for random objects in a category in as great of generality as possible, with the intention of allowing these results to easily translate to a variety of other counting functions in other settings. In a forthcoming paper \cite{alberts2023}, the author will apply these results to the case of Malle's conjecture to both recreate and improve on existing predictions for the asymptotic growth rate. This will be done by constructing a category of ``groups with local data'', applying the results of \cite{sawin-wood2022} to construct a measure on this category, and applying the results of this paper to prove that 100\% of random groups with local data satisfy Malle's conjecture.

We adopt the notation and terminology of \cite{sawin-wood2022} throughout. Let $C$ be a diamond category in the sense of \cite[Definition 1.3]{sawin-wood2022} with countably many isomorphism classes, and suppose $\mu$ is a probability measure (so the whole space has measure $1$) on the isomorphism classes of corresponding pro-objects, $\mathcal{P}$, under the level topology with finite moments $M_G$ for each $G\in C/\cong$. Sawin--Wood give sufficient conditions for a sequence $M_G$ to correspond to such a measure in \cite[Theorem 1.7 and 1.8]{sawin-wood2022}. Our results do not require $\mu$ to be constructed in this way, but the author expects that most interesting examples will come from the existence results proven by Sawin--Wood.

\subsection{Counting Functions}
To translate the setting given by \cite{sawin-wood2022} to that of counting functions more closely resembling (\ref{eq:counting}), we very generally define an \textbf{ordering} to be a sequence of functions $f_n:C/\cong \to \C$ indexed by positive integers $n$. We are motivated by classical orderings such as the discriminant and height functions, which correspond to $f_n$ being the characteristic function of an increasing chain of finite sets $A_n\subseteq C/\cong$, i.e. a chain of subsets for which $A_{n}\subseteq A_{n+1}$. In the general format described in (\ref{eq:counting}), we would choose the ordering
\[
f_n({\rm object}) = \begin{cases}
1 & {\rm order}({\rm object}) \le n\\
0 & \text{else}.
\end{cases}
\]
The property that such functions have finite support is called the Northcott Property, and is known to hold for discriminant and height orderings among many others.
\begin{definition}\label{def:countingfunction}
The \textbf{counting function} on a pro-object $\mathscr{G}\in \mathcal{P}$ ordered by $f_n:C/\cong \to \C$ is defined by
\begin{align*}
N(\mathscr{G},f_n) = \sum_{G\in C/\cong} f_n(G) \#{\rm Epi}(\mathscr{G},G)
\end{align*}
as a function of $n$, when the series is convergent.
\end{definition}
In the classical setting, a Northcott Property is used to guarantee that the counting function is well-defined by forcing the series to be finite. We remark that in the classical setting ${\rm order}({\rm object})$ is taken to be bounded above by a real number $X$, while Definition \ref{def:countingfunction} restricts to only using integer bounds, $n$. This will not make a difference for classical orderings, as the (norm of the) discriminant and height functions are integer-valued. However, it is an artifact of the methods we use that we ask the sequence $(f_n)_{n=1}^{\infty}$ to be indexed by a \emph{countable} set rather than the uncountable set of positive real numbers.

In our general setting, we may relax the Northcott requirement and still guarantee an (almost everywhere) well-defined counting function. Fix a probability measure $\mu$ on the category of pro-objects $\mathcal{P}$ with finite moments $M_G$. For convenience, we package the sequence $M_G$ as a discrete measure $M:C/\cong\to \R_{\ge 0}$ with $M(\{G\})=M_G$ (note that while we require $\mu$ to be a probability measure, $M$ need not be).
\begin{definition}\label{def:L1ordering}
Fix a probability measure $\mu$ on the category of pro-objects $\mathcal{P}$ with finite moments $M_G$. We call $f_n$ an \textbf{$L^1$-ordering} with respect to $\mu$ if
\[
\int_C |f_n|\ dM = \sum_{G\in C/\cong} |f_n(G)| M_G < \infty.
\]
In other words, the $f_n$ are $L^1$-functions for the discrete measure $M$ induced by $\mu$. We will often omit ``with respect to $\mu$'' if the probability measure is clear from context.
\end{definition}
We will prove in Lemma \ref{lem:nicecase} that $N(\mathscr{G},f_n)$ is well-defined as a function of $n$ for almost all $\mathscr{G}$ whenever $f_n$ is an $L^1$-ordering (i.e., there exists a measure $1$ set of $\mathscr{G}$ on which the counting function is well-defined as a function on the positive integers), and that the expected value is given by
\begin{align}\label{eq:exp}
\int_{\mathcal{P}} N(\mathscr{G},f_n)\ d\mu(\mathscr{G}) = \int_C f_n\ dM.
\end{align}
In the classical case, if $f_n$ has finite support in an increasing chain as $n$ tends towards $\infty$ then the counting function $N(\mathscr{G},f_n)$ is a sum of increasing length depending on $n$, where the summands $f_n(G)\#{\rm Epi}(\mathscr{G},G)$ are random variables as $\mathscr{G}$ varies according to $\mu$. This is precisely the setting of the Law of Large Numbers, which suggests a much stronger result of the form
\[
\frac{N(\mathscr{G},f_n)}{\displaystyle \int_C f_n\ dM} \longrightarrow 1
\]
as $n\to \infty$ with probability $1$ in some appropriate sense. Typically, there is some requirement of pairwise independence between the summands in order to prove such a statement (see \cite{sen-singer1993} for the classical statement, and \cite{seneta2013} for a brief history of the Law of Large Numbers, including cases that relax the requirement of pairwise independence). The work of Korchevsky--Petrov \cite{korchevsky-petrov2010} on nonegative random variables, building on work in \cite{etemadi1983a,etemadi1983b,petrov2009a,petrov2009b}, is particularly relevant to these counting functions, as $\#{\rm Epi}(\mathscr{G},G)$ is always nonegative. We will prove several different versions of the Law of Large Numbers for $N(\mathscr{G},f_n)$ of varying strengths under sufficiently nice orderings.

\subsection{Main Results}

Take $C$ to be a diamond category with countably many isomorphism classes, $\mathcal{P}$ the category of pro-objects, and $\mu$ a probability measure on $\mathcal{P}$ with finite moments $M_G$. Theorem \ref{thm:LLNintro} is our primary probabilistic result, while Theorems \ref{thm:bounding_mixed_moments_2} and \ref{thm:bounding_mixed_moments_2k} are our primary category theoretic results.

\begin{theorem}[Law of Large Numbers in a Category]\label{thm:LLNintro}
Let $f_n:C/\cong\to \R$ be a real-valued $L^1$-ordering for which $\liminf_{n\to \infty} \left\lvert\int_C f_n dM\right\rvert > 0$. Suppose there exists an integer $k$ and a non-decreasing function $\gamma:\N\to \R^+$ for which $\lim_{t\to \infty} \gamma(t) = \infty$, and
\begin{align}\label{eq:main_bound}
\int_{\mathscr{G}}\left\lvert N(\mathscr{G},f_n) - \int_C f_n\ dM\right\rvert^{k}d\mu(\mathscr{G}) = O\left(\frac{\left\lvert\int_C f_n\ dM\right\rvert^k}{\gamma(n)}\right).
\end{align}
Then each of the following hold:
\begin{enumerate}
\item[(i)](Weak Law of Large Numbers)
\[
\frac{N(\mathscr{G},f_n)}{\displaystyle \int_C f_n\ dM} \overset{p.}{\longrightarrow} 1
\]
as $n\to\infty$, where the ``p." stands for converges in probability with respect to $\mu$.
\item[(ii)](Strong Law of Large Numbers) If we additionally assume that $\sum_{n=1}^{\infty}\frac{1}{\gamma(n)} < \infty$ then
\[
\frac{N(\mathscr{G},f_n)}{\displaystyle \int_C f_n\ dM} \overset{a.s.}{\longrightarrow} 1
\]
as $n\to\infty$, where the ``a.s." stands for converges almost surely with respect to $\mu$.
\item[(iii)](Strong Law of Large Numbers) If we additionally assume that
\begin{itemize}
\item $f_n$ is nonegative,
\item the counting function $n\mapsto N(\mathscr{G},f_n)$ is almost everywhere nondecreasing, and
\item $\gamma(n) = \psi(\int_C f_n dM)$ for a nondecreasing function $\psi:\R\to \R^+$ for which $\sum_{n=1}^{\infty}\frac{1}{n\psi(n)} < \infty$,
\end{itemize}
then
\[
\frac{N(\mathscr{G},f_n)}{\displaystyle \int_C f_n\ dM} \overset{a.s.}{\longrightarrow} 1
\]
as $n\to\infty$, where the ``a.s." stands for converges almost surely with respect to $\mu$.
\end{enumerate}
\end{theorem}

The primary new contribution of Theorem \ref{thm:LLNintro} is the connection between counting functions and the Law of Large Numbers. The probabilistic content is largely standard, and in fact the proof of Theorem \ref{thm:LLNintro}(iii) is essentially the same as that of \cite{korchevsky-petrov2010}. We follow a standard technique for bounding probabilities using Chebyshev's law and the first Borel-Cantelli lemma. Most standard references will have some version of this technique, sometimes in the context of the ``method of moments'', such as \cite[Chapter 4]{fischer2011}. This technique is ubiquitous to the point that examples can be found in proofs on Wikipedia \cite{wikipediaLLN} and a number of Math StackExchange posts such as \cite{mathstackLLN1,mathstackLLN2}.

If the $f_n$ are specifically characteristic functions with finite support in an increasing chain then Theorem \ref{thm:LLNintro}(iii) is a special case of \cite[Theorem 3]{korchevsky-petrov2010}. We state and prove these results separately from work in \cite{korchevsky-petrov2010} to allow for a more general class of orderings, although the proof requires no new ideas. For example, we can now directly consider orderings of the form
\[
f_n({\rm object}) = \begin{cases}
1 & n < {\rm order}({\rm object}) \le 2n\\
0 & \text{else}.
\end{cases}
\]
This type of ordering gives counting functions $N(\mathscr{G},f_n)$ that count objects landing in the moving interval $(n,2n]$. These sorts of counting functions are used to avoid errant behavior for small objects, and can be useful for proving averaging results.% The author will take advantage of this generality in forthcoming work \cite{alberts2023}.

In cases for which $\liminf_{n\to \infty}\left\lvert \int_C f_n dM\right\rvert = 0$, there is something we can say using the more general results Theorem \ref{thm:WLLN} and Theorem \ref{thm:SLLN} we prove in the main body of the paper. This requires a finer study of the rate of convergence, as we are no longer allowed to have $\int_C f_n dM$ on the denominator for such cases.

The bound on $k^{\rm th}$ moments in (\ref{eq:main_bound}) is a concrete way to state that the random variables are ``close enough to independent'' for the Law of Large Numbers. Our primary categorical result is on bounding this quantity in terms of purely categorical structures of $C$ and $\mathcal{P}$ together with the discrete moments of the ordering $f_n$.

Given an ordering $f_n:C/\cong \to \C$, we extend the ordering multiplicatively to the product category $f_n:C^k/\cong\to \C$ by
\[
f_n(G_1,...,G_{k}) = f_n(G_1)\cdots f_n(G_{k}).
\]
We also define the discrete measure $M^{(j)}:C^j/\cong \to \R_{\ge 0}\cup\{\infty\}$ given by the mixed moments of the measure $\mu$ as
\[
M_{(G_1,...,G_j)}^{(j)} = \int_{\mathcal{P}} \prod_{i=1}^j\#{\rm Epi}(\mathscr{G},G_i)\ d\mu(\mathscr{G}).
\]
This is sufficient information to prove a bound for (\ref{eq:main_bound}) for $k=2$.

\begin{theorem}\label{thm:bounding_mixed_moments_2}
Let $f_n:C/\cong \to \R$ be a real-valued $L^1$-ordering. Then
\[
\int_{\mathscr{G}}\left\lvert N(\mathscr{G},f_n) - \int_C f_n\ dM\right\rvert^{2}d\mu(\mathscr{G})  \ll_k \max_{j\in\{1,2\}}\int_{C^{2}\setminus E(2,M)} |f_n|\ dM^{(j)}(dM)^{2-j},
\]
where $E(2,M)\subseteq C^2$ is the full subcategory of pairs $(G_1,G_2)\in C^2$ such that
\begin{itemize}
\item[(a)] the epi-product $G_1\times_{\rm epi} G_2$ exists (defined in Definition \ref{def:epiproduct}), and
\item[(b)] $M_{G_1\times_{\rm epi} G_2} = M_{G_1} M_{G_2}$.
\end{itemize}
\end{theorem}

Taking away the full subcategory $E(2,M)$ of $C^{2}$ is the new content, and allows for bounds on the higher moments that can be computed solely from data given by $M$, $C$, and $f_n$ without reference to the measure $\mu$. The epi-product in condition (a) is a special type of product structure in a category, satisfying a similar universal property to the direct product with every morphism in the diagram being an epimorphism (Definition \ref{def:epiproduct}). We will show that the existence of an epi-product together with (b) implies the random variables $\#{\rm Epi}(\mathscr{G},G_1)$ and $\#{\rm Epi}(\mathscr{G},G_2)$ are uncorrelated. This will be key for bounding the moments. It will often be the case that determining the existence or nonexistence of epi-products is easier than approaching the bound (\ref{eq:main_bound}) in Theorem \ref{thm:LLNintro} directly. We include some examples using this approach in Section \ref{sec:examples}.

We will prove a similar bound for (\ref{eq:main_bound}) when $k>2$ is even, but this will require more assumptions on the category. These extra assumptions can prove useful for applying Theorem \ref{thm:bounding_mixed_moments_2}, as they give us extra tools for calculating the integral with respect to $M^{(2)}$. Lemma \ref{lem:MG1Gk} will be very useful in this regard, and we give an example demonstrating its use in Section \ref{sec:examples}.

\begin{theorem}\label{thm:bounding_mixed_moments_2k}
Suppose $C$ is a category with for which every morphism factors uniquely (up to isomorphism) as a composition of an epimorphism with a monomorphism. Let $f_n:C/\cong \to \R$ be a real-valued $L^1$-ordering and $k$ a positive integer. Then
\[
\int_{\mathscr{G}}\left\lvert N(\mathscr{G},f_n) - \int_C f_n\ dM\right\rvert^{2k}d\mu(\mathscr{G})  \ll_k \max_{j\in\{1,...,2k\}}\int_{C^{2k}\setminus E(2k,M)} |f_n|\ dM^{(j)}(dM)^{2k-j},
\]
where $E(2k,M)\subseteq C^{2k}$ is the full subcategory of tuples $(G_1,G_2,...,G_{2k})$ for which there exists an index $i$ such that for each subset $A\subseteq\{1,...,k\}$ with $i\not\in A$
\begin{itemize}
\item[(a)] the product object $G_A:=\prod_{m\in A} G_m$ exists in $C^{2k}$,
\item[(b)] $G_A$ has finitely many subobjects up to isomorphism,
\item[(c)] the epi-product $G_i\times_{\rm epi} H$ exists for each subobject $H\hookrightarrow G_A$ (defined in Definition \ref{def:epiproduct}), and
\item[(d)] $M_{G_i\times_{\rm epi} H} = M_{G_i} M_{H}$ for each subobject $H\hookrightarrow G_A$.
\end{itemize}
\end{theorem}

Unique factorization of morphisms is known for many categories of interest, including finite sets, finite groups, and finite modules over a ring. In fact, each of these categories automatically satisfies parts (a) and (b) as well. Thus, in numerous categories of interest, Theorem \ref{thm:bounding_mixed_moments_2k} requires no extra input over Theorem \ref{thm:bounding_mixed_moments_2}.

We also prove an upper bound utilizing simpler categorical structures that holds without the restrictions of Theorem \ref{thm:LLNintro}.

\begin{corollary}\label{cor:LLNintrobound}
Let $C$ be a diamond category which has countably many isomorphism classes, $\mu$ be a probability measure on the isomorphism classes of the corresponding category of pro-objects $\mathcal{P}$ with finite moments $M_G$, and $f_n:C/\cong \to \C$ an $L^1$-ordering.

Then for any $\epsilon > 0$ and any positive integer $k$
\[
\frac{N(\mathscr{G},f_n)}{\displaystyle n^{\frac{1+\epsilon}{k}}\max_{j\in\{1,...,k\}}\left\{\int_{C^{k}} |f_n|\ dM^{(j)}(dM)^{k-j}\right\}^{1/k}} \overset{a.s.}{\longrightarrow} 0
\]
as $n\to\infty$, where the ``a.s." stands for converges almost surely with respect to $\mu$.
\end{corollary}

Corollary \ref{cor:LLNintrobound} is motivated by Malle's predicted weak upper bound, although the correspondence is less than obvious. In the classical case $f_n$ is a nonegative function so that $\int f_n\ dM = \int |f_n|\ dM$, and when $j=1$ the denominator is given by
\[
n^{\frac{1+\epsilon}{k}}\int_C f_n\ dM.
\]
Ignoring the other values of $j$ and taking $k$ sufficiently large, Corollary \ref{cor:LLNintrobound} appears to give an upper bound of the form
\[
N(\mathscr{G},f_n)\ll n^{\epsilon}\int_C f_n\ dM
\]
with probability $1$, where we recall that $n$ is playing the role of $X$ in Malle's counting function. Of course, we cannot actually ignore the other values of $j$, and depending on $C$ and $M$ it is possible for these bounds to be worse.% This will be a point of considerable interest in forthcoming work \cite{alberts2023}.

Theorem \ref{thm:LLNintro} and Corollary \ref{cor:LLNintrobound} can be used to provide evidence for conjectures about counting functions in great generality. We state this as an explicit heuristic:

\begin{heuristic}[Vast Counting Heuristic]\label{heur:vast}
Let $C$ be a diamond category which has countably many isomorphism classes and let $\mu$ be a probability measure on the isomorphism classes of the corresponding category of pro-objects $\mathcal{P}$ with finite moments $M_G$. Choose an $L^1$-ordering $f_n$.

Let $\mathbb{G}$ be an element of $\mathcal{P}/\cong$ for which $N(\mathbb{G},f_n)$ is well-defined. If $\mathbb{G}$ is an object we expect to behave typically$^{*}$ with respect to the probability measure $\mu$, such as an object coming from arithmetic, then
\begin{itemize}
\item[(i)] (Strong form) If $f_n$ satisfies the bound in (\ref{eq:main_bound}) for some positive integer $k$, then
\[
N(\mathbb{G},f_n) \sim \int_C f_n\ dM
\]
as $n\to \infty$.
\item[(ii)] (Weak form) For each $\epsilon > 0$ and each positive integer $k$
\[
N(\mathbb{G},f_n) \ll n^{\frac{1+\epsilon}{k}}\max_{j\in \{0,...,k\}}\left(\int_{C^{j}} |f_n|\ dM^{(j)}(dM)^{k-j}\right)^{1/k}.
\]
\end{itemize}
\end{heuristic}

The asterisk on ``typically'' is meant to highlight that this word is where the intricacies of counting conjectures lie, and in particular is in recognition of counter-examples to Malle's conjecture and the like. For number field counting, Malle's original conjecture is incorrect due to behavior coming from roots of unity that Malle did not initially consider. For the purpose of a general statement, this can be interpreted as saying ``if there are no obstructions, this is what we expect". Heuristic \ref{heur:vast} provides a baseline for counting predictions when we believe there are no ``non-random" structures of $\mathbb{G}$ missing from the measure $\mu$ on the category $\mathcal{P}$. Issues, like those occurring in Malle's conjecture from roots of unity, might be fixed by either (a) incorporating the missing structure into a modified version of the category (this is analogous to the approach commonly take in the study of Malle's conjecture), or (b) excluding epimorphisms from the count which are affected by the missing structure (this is the approach taken for the Batyrev-Manin conjecture).

%We remark that for the weak form of Malle's conjecture, there are no known issues. No counter-examples are known for the weak form of Malle's conjecture, so a part of forthcoming work of the author \cite{alberts2023} will be showing that Heuristic \ref{heur:vast}(ii) for Malle's counting function provides extremely strong evidence coming from Corollary \ref{cor:LLNintrobound}.

The literature on proving new cases of the Law of Large Numbers is vast, and it is possible that Theorem \ref{thm:LLNintro} may be extended with finer probabilistic arguments. From a categorical perspective, we've stated Theorem \ref{thm:bounding_mixed_moments_2}, Theorem \ref{thm:bounding_mixed_moments_2k}, Corollary \ref{cor:LLNintrobound} and Heuristic \ref{heur:vast} in the greatest generality possible so that they might be applied to make predictions on a variety of different counting functions in arithmetic statistics independent of the author's immediate knowledge. We hope that these results will be useful as a starting point for many future predictions for the asymptotic growth rates of counting functions.

\subsection{Layout of the Paper}

We begin in Section \ref{sec:defLLN} with a discussion of what we mean by an asymptotic version of the Law of Large Numbers. This is conceptually important for understanding the methods of this paper, as the references to the Law of Large Numbers are not just in analogy. Korchevskey--Petrov's result \cite[Theorem 3]{korchevsky-petrov2010} is a particular example of the type of Law of Large Numbers result we consider.

We prove well-definedness of the counting function in Section \ref{sec:countingfunction}. We give the probabilistic proofs of Theorem \ref{thm:LLNintro} and Corollary \ref{cor:LLNintrobound} in Section \ref{sec:LLNstate}. These proofs closely follow the methods of \cite{korchevsky-petrov2010}. We define epi-products in Section \ref{sec:epi-prod}, prove some useful features of these objects, and prove Theorems \ref{thm:bounding_mixed_moments_2} and \ref{thm:bounding_mixed_moments_2k}.

Lastly, we include Section \ref{sec:examples} to work out a couple of important special cases. The arithmetic statistics questions that primarily motivated the author will require independent work to set up the necessary category, measure, and orderings which the author will present in a forthcoming paper \cite{alberts2023}. We instead include simpler examples where these steps are either shorter or done by previous work such as in \cite{sawin-wood2022}, and we pay particular attention to when epi-products exist in such cases. Despite the easier nature of these examples, we recover existing important random models from number theory as special cases of Theorem \ref{thm:LLNintro} and Heuristic \ref{heur:vast}.

\section*{Acknowledgments}

The author would like to thank Melanie Matchett Wood for feedback and helpful discussions on the mechanics on \cite{sawin-wood2022}.

\section{An asymptotic version of the Law of Large Numbers}\label{sec:defLLN}
Given a sequence of independent, identically distributed random variables $\{X_i\}$ the Law of Large Numbers says that we expect $\frac{1}{n}(X_1+\cdots+X_n)$ to converge to the expected value of $X_i$ in some sense. A version of this statement can sometimes also hold for random variables which are neither independent nor identically distributed. If the variables are not identically distributed, then $\mathbb{E}[X_i]$ may vary depending on $i$, which necessarily alters the statement we want to make. We will instead be comparing
\begin{align*}
\sum_{i=1}^n X_i &&\text{with}&& \sum_{i=1}^n\mathbb{E}[X_i]
\end{align*}
as $n\to \infty$. When the random variables are pairwise independent, identically distributed, and have bounded variance in some sense, the difference of these is known to grow slower than $n$ almost surely by Kolmogorov's strong law \cite{etemadi1981,sen-singer1993}. Rather than a convergence statement, we will state the Law of Large Numbers as an asymptotic statement. The conclusion of Kolmogorov's strong law would then be written as
\[
{\rm Prob}\left(\frac{1}{n}\sum_{i=1}^n X_i - \frac{1}{n}\sum_{i=1}^n\mathbb{E}[X_i] = o(1)\text{ as }n\to \infty\right) = 1.
\]
We can understand this term from the perspective of asymptotic vounting functions as picking out a ``main term + O(error term)''. Kolmogorov's strong law can then be understood as
\[
\sum_{i=1}^n X_i = \sum_{i=1}^n \mathbb{E}[X_i] + o(n)
\]
as $n\to \infty$ with probability $1$. It is often the case with the Law of Large Numbers that such results are written as
\[
\frac{\sum_{i=1}^n X_i  - \sum_{i=1}^n \mathbb{E}[X_i]}{n} \overset{a.s.}{\longrightarrow} 0.
\]
The statements of Theorem \ref{thm:LLNintro} can be interpretted as something similar to this. $N(\mathscr{G},f_n)$ is defined to be a sum of random variables $f_n(G)\#{\rm Epi}(\mathscr{G},G)$, so if we label $X_G = f_n(G)\#{\rm Epi}(\mathscr{G},G)$ we can understand the counting function as
\[
N(\mathscr{G},f_n) = \sum_{G, f_n(G) \ne 0} X_G.
\]
The ordering is analogous to a statement of the form ``${\rm order}(G)<n$'', although our results apply to more general orderings. We can then see that $N(\mathscr{G},f_n)$ behaves like a sum of random variables $X_G$, whose length is growing with $n$. The philosophy of the Law of Large Numbers suggests that we should expect this to be close to $\sum_{G, f_n(G) \ne 0} \mathbb{E}[X_G] = \int_C f_n\ dM$. However, there is no reason to expect an error term like $o(n)$ to hold. This isn't precisely what we want, we are interested in a statement of the form
\[
N(\mathscr{G},f_n) = \int_C f_n\ dM + o\left(\int_C f_n\ dM\right)
\]
with probability $1$. This error does not necessarily agree with $o(n)$.

We consider two examples to demonstrate the intricacies of the error term.
\begin{itemize}
\item Consider the following example: Let $X$ be a random variable equal to $1$ with probability $1/2$ and $0$ with probability $1/2$, and consider the dependent sequence $X_n = \frac{1}{n}X$. Classical SLLN certainly holds for this sequence, as
\begin{align*}
\left\lvert\sum_{i=1}^n X_i - \sum_{i=1}^n \mathbb{E}[X_i]\right\rvert &\le \sum_{i=1}^n X_i + \sum_{i=1}^n \frac{1}{2i}\\
&\le \sum_{i=1}^n \frac{1}{i} + \sum_{i=1}^n \frac{1}{2i}\\
&=\sum_{i=1}^n \frac{1}{3i}\\
&= O(\log n) = o(n).
\end{align*}
However, the $X_n$ are identically 0 with probability $1/2$, so
\[
{\rm Prob}\left(\sum _{i=1}^n X_i \sim \sum_{i=1}^n \mathbb{E}[X_i]\text{ as }n\to \infty\right) \le \frac{1}{2}.
\]
This example satisfies the $o(n)$ error term, but it is certainly not what we want. What is happening here is that
\[
\sum_{i=1}^n \mathbb{E}[X_i] = o(n)
\]
right from the start, so the classical error for the Law of Large Numbers does not actually tell us much about how close $\sum X_i$ is to $\sum \mathbb{E}[X_i]$.

\item If we instead work backwards, we might ask for error terms which are smaller than
\[
o\left(\sum_{i=1}^n \mathbb{E}[X_i]\right),
\]
so that the sum of expected values can rightly be recognized as the main term. In the classical setting with identically distributed random variables with nonzero expected value $E$, these notions coincide as
\[
o\left(\sum_{i=1}^n \mathbb{E}[X_i] \right) = o(En) = o(n).
\]
In fact, this is known more generally. Korchevsky--Petrov \cite[Theorem 3]{korchevsky-petrov2010} give sufficient conditions for when
\[
\frac{S_n - \mathbb{E}[S_n]}{\mathbb{E}[S_n]} \overset{a.s.}{\longrightarrow}1
\]
as $n\to \infty$, when $S_n = \sum_{i=1}^n X_n$ is a sum of \emph{nonegative} random variables. Theorem \ref{thm:LLNintro}(iii) is based on their result.

However, in the case that $\mathbb{E}[X_i] = 0$ for all $i$ this is asking for a trivial error term and a trivial main term. This is certainly an issue, as in the classical case with a sequence of independent identically distributed random variables with mean $0$ the central limit theorem forces the error term to be about $\sqrt{n}$ (i.e., not $o(0)$). The case that all expected values are zero is important in the classical study and applications of the Law of Large Numbers, so we do not want to exclude this case.
\end{itemize}

The condition $\liminf_{n\to \infty}\left\lvert\int_C f_n dM\right\rvert > 0$ in Theorem \ref{thm:LLNintro} is a consequence of the second bullet above. This corresponds to the condition that $\mathbb{E}[S_n] \ne 0$ for all but finitely many $n$, so that we can make sense of having $\mathbb{E}[S_n]$ on the denominator. This issue can alternatively be fixed avoided by relaxing what we want for an error term. Some options might include $o\left(\sum_{i=1}^n \mathbb{E}[|X_i|]\right)$ to guarantee the sum is nonzero unless $X_i$ are identically zero, or something of the form $o\left(n^{1/2+\epsilon}\right)$ to give an error of a similar nature to the Central Limit Limit. We give a version of Theorem \ref{thm:LLNintro}(i,ii) in Theorems \ref{thm:WLLN} and \ref{thm:SLLN} with explicit rates of convergence. These results can be used to produce probability $1$ statements even in the case that $\liminf \left\lvert\int_C f_n dM\right\rvert = 0$, although they will no longer be of the form
\[
N(\mathscr{G},f_n)\sim \int_C f_n dM.
\]

\section{Well-defined counting functions}\label{sec:countingfunction}

A sequence of such $L^1$-functions $f_n$ is called an \textbf{$L^1$-ordering} as in Definition \ref{def:L1ordering}. As in Definition \ref{def:countingfunction}, we define the counting function
\[
N(\mathscr{G},f_n) = \sum_{G\in C/\cong} f_n(G) \#{\rm Epi}(\mathscr{G},G)
\]
for each $\mathscr{G}\in \mathcal{P}$. By allowing $f_n$ to have possibly infinite support, we need to make sure that the counting function is well-defined and the linearity of expectation behaves as we expect. In the case of $L^1$-orderings, we prove this below:

\begin{lemma}\label{lem:nicecase}
If $f_n$ is an $L^1$-ordering, then
\begin{itemize}
\item[(a)] $N(\mathscr{G},f_n)$ is well-defined as a function on the positive integers $n$ almost surely (i.e. with probability $1$) with respect to $\mu$, and
\item[(b)] $\displaystyle\int_{\mathcal{P}} N(\mathscr{G},f_n)\ d\mu(\mathscr{G}) = \int_C f_n\ dM.$
\end{itemize}
\end{lemma}
\begin{proof}
This follows from the Fubini-Tonelli theorem. Indeed, $f_n$ being $L^1$ means
\[
\sum_C \int_{\mathcal{P}} |f_n(G)| \#{\rm Epi}(\mathscr{G},G)\ d\mu(\mathscr{G}) < \infty.
\]
All probability spaces are $\sigma$-finite, so it follows from the Fubinit-Tonelli theorem that
\begin{align*}
\int_{\mathcal{P}} \left\lvert\sum_{G\in C/\cong} |f_n(G)|\#{\rm Epi}(\mathscr{G},G)\right\rvert\ d\mu(\mathscr{G}) &= \int_{\mathcal{P}} \sum_{G\in C/\cong} |f_n(G)|\#{\rm Epi}(\mathscr{G},G)\ d\mu(\mathscr{G})\\
&= \sum_{G\in C/\cong} \int_{\mathcal{P}} |f_n(G)|\#{\rm Epi}(\mathscr{G},G)\ d\mu(\mathscr{G})\\
&=\sum_{G\in C/\cong} |f_n(G)|M_G < \infty.
\end{align*}
Certainly $\mathbb{E}[|X|] < \infty$ implies $X<\infty$ almost surely, so it must be that for fixed $n$ the counting function
\[
N(\mathscr{G},f_n) = \sum_{G\in C/\cong} f_n(G)\#{\rm Epi}(\mathscr{G},G)
\]
converges absolutely for almost all $\mathscr{G}$. As there are countably many $n$, countable additivity implies $N(\mathscr{G},f_n)$ is well-defined almost surely as a function on the positive integers.

The fact that the integrals of $|f_n|$ are finite also implies
\[
\int_{\mathcal{P}} \sum_C f_n(G)\#{\rm Epi}(\mathscr{G},G)\ d\mu(\mathscr{G}) = \sum_C \int_{\mathcal{P}} f_n(G)\#{\rm Epi}(\mathscr{G},G)\ d\mu(\mathscr{G})
\]
by the Fubini-Tonelli Theorem. Evaluating the inner sum/integral gives
\[
\int_{\mathcal{P}} N(\mathscr{G},f_n)\ d\mu = \sum_C f_n(G)M_G = \int_C f_n\ dM \le \int_C |f_n|\ dM < \infty,
\]
proving (b).
\end{proof}

\section{Law of Large Numbers for categories}\label{sec:LLNstate}

We prove Theorem \ref{thm:LLNintro} and Corollary \ref{cor:LLNintrobound} in this section. The probabilistic content of each proof is essentially Chebyshev's bound, the Borel-Cantelli lemma, and a few tricks for nonegative orderings previously applied in \cite{korchevsky-petrov2010}. We prove Theorem \ref{thm:LLNintro}(i,ii) with explicit control on the rate of convergence, in part to make the proof of Corollary \ref{cor:LLNintrobound} more straight forward.

\subsection{The Weak Law of Large Numbers}

We prove the following result, which is Theorem \ref{thm:LLNintro}(i) together with an explicit rate of convergence.

\begin{theorem}\label{thm:WLLN}
Let $f_n:C/\cong\to \R$ be a real-valued $L^1$-ordering. Suppose there exists an integer $k$, a non-decreasing function $\gamma:\N\to \R^+$ for which $\lim_{t\to \infty} \gamma(t) = \infty$, and a function $E:\N\to \R^+$ for which
\begin{align}
\int_{\mathscr{G}}\left\lvert N(\mathscr{G},f_n) - \int_C f_n\ dM\right\rvert^{k}d\mu(\mathscr{G}) = O\left(\frac{E(n)^k}{\gamma(n)}\right).
\end{align}
Then
\[
\frac{N(\mathscr{G},f_n) - \displaystyle \int_C f_n\ dM}{E(n)} \overset{p.}{\longrightarrow} 0
\]
as $n\to\infty$, where the ``p." stands for converges in probability with respect to $\mu$.
\end{theorem}

Theorem \ref{thm:LLNintro}(i) follows by taking $E(n) = \max\left\{\left\lvert\int_C f_n\ dM\right\rvert,\delta\right\}$ for some small $\delta > 0$.

\begin{proof}
If $Y$ is a random variable with $\mathbb{E}[Y] = 0$, then Chebyshev's bound states that for each positive integer $k$,
\[
{\rm Prob}\left(|Y| > \lambda\right) \le \lambda^{-k}\mathbb{E}[|Y|^{k}].
\]
(See, for instance, \cite{mitzenmacher-upfal2017}.) In the context of this proof, we let $\mathbb{E}$ denote the expected value with respect to $\mu(\mathscr{G})$ for the sake of convenience. Set
\[
Y_n = N(\mathscr{G},f_n) - \int_C f_n\ dM.
\]
It then follows that
\begin{align*}
\mathbb{E}[|Y_n|^{k}] &= \int_{\mathcal{P}} \left\lvert N(\mathscr{G},f_n) - \int_C f_n\ dM\right\rvert^{k}\ d\mu(\mathscr{G})\\
&=O\left(\frac{E(n)^k}{\gamma(n)}\right).
\end{align*}
Taking $\lambda = \epsilon E(n)$, it follows that
\begin{align*}
{\rm Prob}\left(|Y_n| > \epsilon E(n)\right) \ll \frac{1}{\epsilon^k \gamma(n)}
\end{align*}
for each $\epsilon > 0$. Thus,
\begin{align*}
\limsup_{n\to\infty} {\rm Prob}\left(\frac{|Y_n|}{E(n)} > \epsilon\right) = 0.
\end{align*}
This is the definition of $Y_n/E(n)$ converging in probability to $0$, and so concludes the proof.
\end{proof}

\subsection{The Strong Law of Large Numbers}

We prove the following result, which is Theorem \ref{thm:LLNintro}(ii) together with an explicit rate of convergence.

\begin{theorem}\label{thm:SLLN}
Let $f_n:C/\cong\to \R$ be a real-valued $L^1$-ordering. Suppose there exists an integer $k$, a non-decreasing function $\gamma:\N\to \R^+$ for which $\lim_{t\to \infty} \gamma(t) = \infty$, and a function $E:\N\to \R^+$ for which
\begin{align}
\int_{\mathscr{G}}\left\lvert N(\mathscr{G},f_n) - \int_C f_n\ dM\right\rvert^{k}d\mu(\mathscr{G}) = O\left(\frac{E(n)^k}{\gamma(n)}\right).
\end{align}
Additionally, assume that $\sum_{n=1}^{\infty} \frac{1}{\gamma(n)} < \infty$. Then
\[
\frac{N(\mathscr{G},f_n) - \displaystyle \int_C f_n\ dM}{E(n)} \overset{a.s.}{\longrightarrow} 0
\]
as $n\to\infty$, where the ``a.s." stands for converges almost surely with respect to $\mu$.
\end{theorem}

Theorem \ref{thm:LLNintro}(ii) follows by taking $E(n) = \max\left\{\left\lvert\int_C f_n\ dM\right\rvert,\delta\right\}$ for some small $\delta > 0$.

\begin{proof}
We proceed in a similar way to the proof of Theorem \ref{thm:WLLN}. Set
\[
Y_n = N(\mathscr{G},f_n) - \int_C f_n\ dM
\]
so that
\begin{align*}
\mathbb{E}[|Y_n|^{k}] &= \int_{\mathcal{P}} \left\lvert N(\mathscr{G},f_n) - \int_C f_n\ dM\right\rvert^{k}\ d\mu(\mathscr{G})\\
&=O\left(\frac{E(n)^k}{\gamma(n)}\right).
\end{align*}
It then follows from Chebyshev's bound that
\[
{\rm Prob}\left(|Y_n| > \lambda\right) \ll \frac{E(n)^k}{\lambda^k\gamma(n)}.
\]
Taking $\lambda = \epsilon E(n)$, it follows that
\begin{align*}
\sum_{n=1}^{\infty}{\rm Prob}\left(|Y_n| > \epsilon E(n)\right) \ll \sum_{n=1}^{\infty}\frac{1}{\epsilon^k \gamma(n)} < \infty.
\end{align*}
for each $\epsilon > 0$. Thus, the Borel-Cantelli lemma implies
\[
{\rm Prob}\left(|Y_n| > \epsilon E(n)\text{ infinitely often }\right) = 0,
\]
so that
\begin{align*}
{\rm Prob}\left(\limsup_{n\to\infty} \frac{|Y_n|}{E(n)} > \epsilon\right) = 0.
\end{align*}
By countable additivity and taking $\epsilon = 1/m$ for $m\in \Z$ tending towards $\infty$, this implies $Y_n/E(n)$ converges to $0$ almost surely, concluding the proof.
\end{proof}

\subsection{Almost Sure Upper Bounds}

Corollary \ref{cor:LLNintrobound} follows almost immediately from Theorem \ref{thm:SLLN}.

\begin{proof}[Proof of Corollary \ref{cor:LLNintrobound}]
We bound
\begin{align*}
&\int_{\mathscr{G}}\left\lvert N(\mathscr{G},f_n) - \int_C f_n\ dM\right\rvert^{k}d\mu(\mathscr{G})\\
&\le \int_{\mathscr{G}}\left(|N(\mathscr{G},f_n)| + \int_C |f_n|\ dM\right)^{k}d\mu(\mathscr{G})\\
&\le \int_{\mathscr{G}}\left(\sum_{G\in C/\cong} |f_n(G)|(\#{\rm Epi}(\mathscr{G},G) + M_G)\right)^{k}d\mu(\mathscr{G})\\
&\le\sum_{(G_1,...,G_k)\in C^k/\cong} |f_n(G_1,...,G_k)| \sum_{j=0}^{2k} \sum_{\sigma\in S_k} M_{(G_{\sigma(1)},...,G_{\sigma(j)})}^{(j)} M_{G_{\sigma(j+1)}}\cdots M_{G_{\sigma(k)}},
\end{align*}
where the last line follows from $f_n$ being an $L^1$ function and moving the integral all the way to the inside. We remark that the sum over permutations $\sigma\in S_k$ is a bit larger than the truth, but this bound will be sufficient for our purposes. This is exactly a sum of mixed moments
\begin{align*}
&\le \sum_{j=0}^{2k} \sum_{\sigma\in S_k} \int_{\sigma(C^{k})} |f_n|\ dM^{(j)}(dM)^{k-j}\\
&\ll_k \max_{j\in \{0,1,...,k\}}\left\{\int_{C^{k}} |f_n|\ dM^{(j)}(dM)^{k-j}\right\},
\end{align*}
noting that $f_n$ is invariant under the permuting the coordinates. Taking
\[
E(n) = n^{\frac{1+\epsilon}{k}}\max_{j\in \{0,1,...,k\}}\left\{\int_{C^{k}} |f_n|\ dM^{(j)}(dM)^{k-j}\right\}^{1/k}
\]
and $\gamma(n) = n^{1+\epsilon}$ in Theorem \ref{thm:SLLN} is sufficient to conclude the proof.
\end{proof}

\subsection{The Strong Law of Large Numbers for nonnegative counting functions}

Theorem \ref{thm:LLNintro}(iii) takes advantage of some tricks employed by \cite{korchevsky-petrov2010} for nonegative functions. The following proof is in essense the same as their main result, however the functions $f_n(G)$ we allow are slightly more general than the sequence of weights $w_k$ utilized in \cite[Theorem 1]{korchevsky-petrov2010}.

These tricks are not compatible with controlling the rate of convergence, so we do not state a more general result for this part. It is certainly possible that more can be proven for this case but we do not do so here.

\begin{proof}[Proof of Theorem \ref{thm:LLNintro}(iii)]
We remark that $n\mapsto N(\mathscr{G},f_n)$ being nondecreasing almost everywhere implies
\[
n\mapsto \int_C f_n\ dM = \int_{\mathcal{P}} N(\mathscr{G},f_n)\ d\mu(\mathscr{G})
\]
is also nondecreasing.

The primary trick is to prove convergence along a subsequence of indices first, then interpolate to the remaining indices.

Fix $b>1$. Define the subsequence of natrual numbers $(n_i)$ by
\[
n_i = \inf\left\{ n : \int_C f_n\ dM \ge b^i\right\}.
\]
These necessarily exist by $\int_C f_n dM \to \infty$ with $n$. We apply Chebyshev's inequality to
\[
Y_n = N(\mathscr{G},f_n) - \int_C f_n\ dM
\]
with $\lambda = \epsilon \int_C f_n\ dM$ to prove that
\begin{align*}
\sum_{i=1}^{\infty} {\rm Prob}\left(|Y_{n_i}| > \epsilon\int_C f_{n_i}\ dM\right) &\ll \sum_{i=1}^{\infty} \frac{1}{\epsilon^k \psi\left(\int_C f_n dM\right)}\\
&\le \sum_{i=1}^{\infty} \frac{1}{\epsilon^k \psi\left(b^i\right)},
\end{align*}
where the second comparison follows from $\psi$ being nondecreasing. A straight foward exercise in the convergence of infinite series shows that if $\sum_{n=1}^{\infty} \frac{1}{n\psi(n)}$ converges, then so does $\sum_{i=1}^{\infty} \frac{1}{\psi(b^i)}$ for any $b>1$ (for a reference, see \cite[Lemma 1]{petrov2009a}). Similar to the proof of Theorem \ref{thm:SLLN}, the Borel-Cantelli Lemma then implies
\[
{\rm Prob}\left(\lim_{i\to \infty} \frac{|Y_{n_i}|}{\displaystyle \int_C f_{n_i}\ dM} = 0\right) = 1.
\]

Next, we interpolate to other indices $n$ not belonging to the subsequence $(n_i)$. There exists an $i$ for which $n_i \le n \le n_{i+1}$. We expand
\begin{align*}
\frac{N(\mathscr{G},f_n) - \int_C f_n dM}{\int_C f_n dM} &= \frac{N(\mathscr{G},f_n) - \int_C f_{n_{i+1}} dM}{\int_C f_n dM} + \frac{\int_C f_{n_{i+1}} dM - \int_C f_{n} dM}{\int_C f_n dM}.
\end{align*}
Given that the counting function and its moments are nondecreasing almost everywhere, this produces an almost everywhere upper bound
\begin{align*}
\le \frac{\left\lvert N(\mathscr{G},f_{n_{i+1}}) - \int_C f_{n_{i+1}} dM\right\rvert}{\int_C f_{n_{i+1}} dM}\frac{\int_C f_{n_{i+1}} dM}{\int_C f_{n_{i}} dM} + \frac{\int_C f_{n_{i+1}} dM - \int_C f_{n_{i}} dM}{\int_C f_{n_i} dM}.
\end{align*}
The subsequence $(n_i)$ is nondecreasing, and if $n_i < n_{i+1}$ it follows that
\[
b^i\le \int_C f_{n_{i}} dM \le \int_C f_{n_{i+1}-1} dM < b^{i+1} \le \int_C f_{n_{i+1}} dM.
\]
We can simplify the upper bound to
\begin{align*}
\le \frac{\left\lvert N(\mathscr{G},f_{n_{i+1}}) - \int_C f_{n_{i+1}} dM\right\rvert}{\int_C f_{n_{i+1}} dM}b^2 + (b^2 - 1).
\end{align*}
Taking the limit as $i\to \infty$ implies
\[
\limsup_{n\to \infty} \frac{N(\mathscr{G},f_n) - \int_C f_n dM}{\int_C f_n dM}\le b^2 - 1.
\]
Taking $b>1$ sufficiently close to $1$ gives the desired $\limsup$. The $\liminf$ is calculated similarly using the lower bound
\begin{align*}
\frac{N(\mathscr{G},f_n) - \int_C f_n dM}{\int_C f_n dM} &\ge \frac{N(\mathscr{G},f_{n_i}) - \int_C f_{n_{i}} dM}{\int_C f_{n_i} dM}\frac{\int_C f_{n_i} dM}{\int_C f_{n_{i+1}} dM} - \frac{\int_C f_{n} dM - \int_C f_{n_{i}} dM}{\int_C f_n dM}\\
&\ge \frac{N(\mathscr{G},f_{n_i}) - \int_C f_{n_{i}} dM}{\int_C f_{n_i} dM}b^{-2} - 1 + \frac{\int_C f_{n_{i}} dM}{\int_C f_n dM}\\
&\ge \frac{N(\mathscr{G},f_{n_i}) - \int_C f_{n_{i}} dM}{\int_C f_{n_i} dM}b^{-2} - 1 + b^{-2}.
\end{align*}
The limit as $i\to \infty$ can again be made arbitrarily close to $0$ by taking $b$ close to $1$.
\end{proof}

\section{Epi-Products}\label{sec:epi-prod}

We let $C$ be a diamond category as defined in \cite[Definition 1.3]{sawin-wood2022} with $\mathcal{P}$ the corresponding category of pro-objects. We fix throughout a probability measure $\mu$ on $\mathcal{P}$ with finite moments $M_G$.

Theorem \ref{thm:bounding_mixed_moments_2} and Theorem \ref{thm:bounding_mixed_moments_2k} rely on the concept of an epi-product, which is a categorical construction capturing the lack of correlation between $\#{\rm Epi}(\mathscr{G},G)$ and $\#{\rm Epi}(\mathscr{G},H)$ as $\mathscr{G}$ varies according to $\mu$.

\subsection{The Definition of an Epi-Product}

\begin{definition}\label{def:epiproduct}
We define the \textbf{epi-product} of $G_1$ and $G_2$ to be an object $G_1\times_{\rm Epi} G_2$ with epimorphisms to $G_1$ and $G_2$ which satisfies the universal property
\[
\begin{tikzcd}
H \rar[two heads]\dar[two heads]\drar[two heads, dashed] & G_2 \\
G_1& G_1\times_{\rm Epi} G_2 \lar[two heads]\uar[two heads]
\end{tikzcd}
\]
where all morphisms in the diagram (including the universal morphism) are required to be epimorphisms.
\end{definition}
This is similar to the definition of a product, except we require all the morphisms (including the universal morphism) to be epimorphisms. Notice that we do not ask that $C$ has epi-products in general. This flexibility will allow us to choose a wider variety of categories to work over. We remark that if both the usual product $G_1\times G_2$ and $G_1\times_{\rm Epi} G_2$ exist then there is necessarily a unique isomorphism between them. However, the existence of one does not imply the existence of the other. We will discuss some basic examples at the end of the paper.

\begin{lemma}\label{lem:uncorrelated}
Let $\mu$ be a probability measure on the pro-objects of $C$ with finite moments $M_G$, and let $\mathscr{G}$ vary with respect to $\mu$. Then the random variables $\#{\rm Epi}(\mathscr{G},G)$ and $\#{\rm Epi}(\mathscr{G},H)$ for $G,H\in C/\cong$ are uncorrelated if $G\times_{\rm Epi} H$ exists and $M_{G\times_{\rm Epi} H} = M_G M_H$.
\end{lemma}

%An analogy for this lemma to consider is setting for number fields. The two objects $G$ and $H$ are, broadly speaking, analogous to a $G$-extension $L_1/K$ and an $H$-extension $L_2/K$. The condition that $G\vee H=G\times H$ implies that $G\times H = \Gal(L_1L_2)$, which corresponds to assuming $L_1\cap L_2 = K$, i.e. that their intersection is as small as possible. Chebotarev density then implies that the distributions of Frobenius over $L_1$ and $L_2$ are independent. The lemma is modeled off of this phenomenon, as the Frobenius distributions are analogous to the distributions of these random variables.

\begin{proof}
By the definition of uncorrelated, it suffices to prove that
\[
\mathbb{E}[\#{\rm Epi}(\mathscr{G},G)\#{\rm Epi}(\mathscr{G},H)]=\mathbb{E}[\#{\rm Epi}(\mathscr{G},G)]\mathbb{E}[\#{\rm Epi}(\mathscr{G},H)].
\]
The universal property of epi-products implies that
\[
\#{\rm Epi}(\mathscr{G},G)\#{\rm Epi}(\mathscr{G},H) = \#{\rm Epi}(\mathscr{G},G\times_{\rm Epi} H).
\]
Therefore
\begin{align*}
\mathbb{E}[\#{\rm Epi}(\mathscr{G},G)\#{\rm Epi}(\mathscr{G},H)] &= \mathbb{E}[\#{\rm Epi}(\mathscr{G},G\times_{\rm Epi} H)]\\
&= M_{G\times_{\rm Epi} H}\\
&=M_GM_H\\
&=\mathbb{E}[\#{\rm Epi}(\mathscr{G},G)]\mathbb{E}[\#{\rm Epi}(\mathscr{G},H)].
\end{align*}
\end{proof}

\subsection{The Proof of Theorem \ref{thm:bounding_mixed_moments_2}}

Theorem \ref{thm:bounding_mixed_moments_2} follows almost immediately from Lemma \ref{lem:uncorrelated}.

\begin{proof}[Proof of Theorem \ref{thm:bounding_mixed_moments_2}]
For simplicity in this proof, we write $\mathbb{E}$ for the expected value with respect to $\mu$. When $f_n$ is real-valued, the square is always positive and we can ignore the absolute value. Thus
\begin{align*}
\int_{\mathscr{G}}\left( N(\mathscr{G},f_n) - \int_C f_n\ dM\right)^{2}d\mu(\mathscr{G})
&=\mathbb{E}\left[\left(\sum_{G\in C/\cong} f_n(G)(\#{\rm Epi}(\mathscr{G},G) - M_G)\right)^{2}\right]\\
&=\sum_{(G_1,G_2)\in C^{2}/\cong} f_n(G_1,G_2) \mathbb{E}\left[\prod_{i=1}^{2}\left(\#{\rm Epi}(\mathscr{G},G_i) - M_{G_i}\right)\right]\\
&=\sum_{(G_1,G_2)\in C^{2}/\cong} f_n(G_1,G_2) \left(M^{(2)}_{(G_1,G_2)} - M_{G_1}M_{G_2}\right).
\end{align*}
By Lemma \ref{lem:uncorrelated}, $(G_1,G_2)\in E(2,M)$ implies $\#{\rm Epi}(\mathscr{G},G_1)$ and $\#{\rm Epi}(\mathscr{G},G_2)$ are uncorrelated. By definition, this is equivalent to $M^{(2)}_{(G_1,G_2)} = M_{G_1}M_{G_2}$. Therefore the sum simplifies to
\begin{align*}
&=\sum_{(G_1,G_2)\in C^{2}\setminus E(2,M)/\cong} f_n(G_1,G_2) \left(M^{(2)}_{(G_1,G_2)} - M_{G_1}M_{G_2}\right)\\
&= \int_{C^2\setminus E(2,M)} f_n\ dM^{(2)} - \int_{C^2\setminus E(2,M)} f_n\ (dM)^2.
\end{align*}
The result then follows from the triangle inequality.
\end{proof}

\subsection{The Proof of Theorem \ref{thm:bounding_mixed_moments_2k}}

Theorem \ref{thm:bounding_mixed_moments_2k} is proven similarly to Theorem \ref{thm:bounding_mixed_moments_2}, although the higher power is trickier to keep track of. In particular, some notion of ``mutually uncorrelated'' random variables is needed to compare the moments of three or more of the $\#{\rm Epi}(\mathscr{G},G_i)$ terms.

The extra categorical requirements in Theorem \ref{thm:bounding_mixed_moments_2k} are used to reframe the mixed moments as first moments, which allows us to more easily capture the required ``mutually uncorrelated''ness using subobjects of product objects. This reframing is also useful for computing the mixed moments, so we state and prove it separately:

\begin{lemma}\label{lem:MG1Gk}
Suppose $C$ is a category with for which every morphism factors uniquely (up to isomorphism) as a composition of an epimorphism with a monomorphism. Let $G_1,...,G_j\in C$ be objects for which the product $G_1\times \cdots \times G_j$ exists in $C$ and has finitely many subobjects up to isomorphism. Then
\[
M^{(j)}_{(G_1,...,G_j)} = \sum_{\substack{\iota:H\hookrightarrow \prod_i G_i\\ \rho_m\iota\text{ is an epimorphism}}} M_H,
\]
where the sum is over subobjects $\iota:H\hookrightarrow \prod_i G_i$ up to isomorphism for which $\rho_m \iota$ is also an epimorphism for each projection morphism $\rho_m:\prod_i G_i \to G_m$.
\end{lemma}

\begin{proof}
By the universal property of the product, there is an embedding
\[
\prod_i {\rm Epi}(\mathscr{G},G_i) \hookrightarrow \Hom\left(\mathscr{G},\prod_i G_i\right).
\]
Partitioning this embedding based on the possible images gives a bijection to a disjoint union
\[
\prod_i {\rm Epi}(\mathscr{G},G_i) \longleftrightarrow \coprod_{\substack{\iota:H\hookrightarrow \prod_i G_i\\\rho_m\iota\text{ is an epimorphism}}}{\rm Epi}(\mathscr{G},H).
\]
Taking cardinalities and integrating with respect to $\mu$ concludes the proof.
\end{proof}

We can now prove Theorem \ref{thm:bounding_mixed_moments_2k}.

\begin{proof}[Proof of Theorem \ref{thm:bounding_mixed_moments_2k}]
For simplicity in this proof, we write $\mathbb{E}$ for the expected value with respect to $\mu$. When $f_n$ is real-valued and we take the $2k^{\rm th}$ moment, we can essentially ignore the absolute values. We then compute
\begin{align*}
&\int_{\mathscr{G}}\left\lvert N(\mathscr{G},f_n) - \int_C f_n\ dM\right\rvert^{2k}d\mu(\mathscr{G})\\
&= \mathbb{E}\left[\left(\sum_{G\in C/\cong} f_n(G)(\#{\rm Epi}(\mathscr{G},G) - M_G)\right)^{2k}\right]\\
&=\sum_{(G_1,...,G_{2k})\in C^{2k}/\cong} f_n(G_1,...,G_k) \mathbb{E}\left[\prod_{i=1}^{2k}\left(\#{\rm Epi}(\mathscr{G},G_i) - M_{G_i}\right)\right],
\end{align*}
where we can switch the order of the integral defining $\mathbb{E}$ and the countable sum because $f_n$ is an $L^1$-ordering.

We multiply out the product of random variables to produce
\begin{align*}
\mathbb{E}\left[\prod_{i=1}^{2k}\left(\#{\rm Epi}(\mathscr{G},G_i) - M_{G_i}\right)\right] &=\sum_{A\subseteq\{1,2,...,k\}} (-1)^{|A|} M^{(|A|)}_{(G_m)_{m\in A}} \prod_{m\not\in A} M_{G_m}.
\end{align*}
It suffices to show this quantity is $0$ when $(G_1,...,G_{2k})\in E(2k,M)$, as the fact that $f_n$, $M^{(j)}$, and $E(2k,M)$ are invariant under permutation implies
\begin{align*}
&\sum_{(G_1,...,G_k)\in C^{2k}\setminus E(2k,M)/\cong} f_n(G_1,...,G_{2k})\sum_{A\subseteq\{1,2,...,k\}} (-1)^{|A|} M^{(|A|)}_{(G_m)_{m\in A}} \prod_{m\not\in A} M_{G_m}\\
&\ll_k \max_{j\in \{1,2,...,2k\}} \sum_{\substack{A\subseteq\{1,2,...,k\}\\|A|=j}} \sum_{(G_1,...,G_k)\in C^{2k}\setminus E(2k,M)/\cong} |f_n(G_1,...,G_{2k})| M^{(j)}_{(G_m)_{m\in A}} \prod_{m\not\in A} M_{G_m}\\
&\ll_k \max_{j\in \{1,...,2k\}} \int_{C^{2k}\setminus E(2k,M)} |f_n|\ dM^{(j)}(dM)^{2k-j}.
\end{align*}

Now, suppose that $(G_1,...,G_{2k})\in E(2k,M)$, with distinguished index $i$. We separate the $i^{\rm th}$ term of the product out, and compute
\begin{align*}
&\mathbb{E}\left[\prod_{m=1}^{2k}\left(\#{\rm Epi}(\mathscr{G},G_m) - M_{G_m}\right)\right]\\
&=\mathbb{E}\left[\#{\rm Epi}(\mathscr{G},G_i)\prod_{m\ne i}\left(\#{\rm Epi}(\mathscr{G},G_m) - M_{G_m}\right)\right] - M_{G_i}\mathbb{E}\left[\prod_{m\ne i}\left(\#{\rm Epi}(\mathscr{G},G_m) - M_{G_m}\right)\right]\\
&=\sum_{\substack{A\subset\{1,...,2k\}\\i\not\in A}}\sum_{\substack{\iota:H\hookrightarrow G_A\\\rho_m\iota\text{ is an epi.}}} \left(\mathbb{E}\left[\#{\rm Epi}(\mathscr{G},G_i)\#{\rm Epi}(\mathscr{G},H)\right]- M_{G_i} M_H \right)\prod_{m\not\in A} M_{G_m}.
\end{align*}
By Lemma \ref{lem:uncorrelated} and the properties of the distinguished $i^{\rm th}$ index, it is necessarily the case that
\[
\mathbb{E}\left[\#{\rm Epi}(\mathscr{G},G_i)\#{\rm Epi}(\mathscr{G},H)\right] = M_{G_i} M_H
\]
for each $H$. Therefore this is a sum of zeros, and cancels out as claimed.
\end{proof}

\section{Examples}\label{sec:examples}

Determining the bound in Theorem \ref{thm:bounding_mixed_moments_2} and Theorem \ref{thm:bounding_mixed_moments_2k} is fairly specific to the category, and the motivating examples of number field counting and the Batyrev-Manin conjecture take a fair bit of set up to define the corresponding diamond category $C$, construct a measure modeling arithmetic behavior, translate classical orderings into a sequence $f_n$, and compute the moments $\int f_n\ dM$. All this before even considering the question of bounding the higher moments. The plus side is the immense return value of Theorem \ref{thm:LLNintro} and Corollary \ref{cor:LLNintrobound}. If you understand when epi-products exist enough to apply these results for one ordering, then it is often the case that the same reasoning will apply to numerous other orderings on the same category.

The author will construct a model for Malle's counting function and determine the reasonableness of a large collection of orderings in forthcoming work \cite{alberts2023}. For the purposes of this paper, we include some more basic examples where these steps are either easier or already done for us in works like \cite{sawin-wood2022}.

\subsection{Random subsets with independent elements}

We prove a typical tool used to give random models for sets of integers as a special case of our methods:
\begin{theorem}
Let $\mathscr{G}$ be a random subset of $\N$ where we let the events $(n\in \mathscr{G})$ be pairwise independent with probability $r_n\in [0,1]$. This forms a true probability measure on $2^{\N}$, and for any any $\epsilon>0$
\[
\frac{\displaystyle \#(\mathscr{G}\cap\{1,...,n\}) - \sum_{j\le n}r_j }{\displaystyle n^{\epsilon}\sqrt{\sum_{j\le n} r_j}} \overset{a.s.}{\longrightarrow} 0
\]
as $n\to \infty$.
\end{theorem}

If we let $r_n = \frac{1}{\log(n)}$ be the probability that $n$ is prime (with $r_1=0$ and $r_2=1$ in order to make sense) this is precisely Cram\'er's original random model for the set of primes \cite{cramer1994}. We reproduce Cram\'er's main term with this asymptotic:
\[
\#(\mathscr{G}\cap\{1,...,n\}) = \sum_{j=3}^n \frac{1}{\log(j)} + o\left(n^{\epsilon}\sqrt{\sum_{j=3}^n \frac{1}{\log(j)}}\right).
\]
with probability $1$. One easily checks using difference calculus that the main term is of the same order of magnitude as ${\rm Li}(n)$ in agreement with the prime number theorem. The error we produce, while not as strong as Cram\'er's, is still $o(n^{1/2+\epsilon})$ suggesting the truth of the Riemann hypothesis.

\begin{proof}
Let $C$ be the category gotten by letting $C^{\rm op}$ be the category of finite subsets of $\N$ whose morphisms are inclusions. The pro-objects are all the subsets of $\N$. This category is incredibly nice for numerous reasons:
\begin{itemize}
\item[(a)] $\Hom(A,B) = {\rm Epi}(A,B)$ contains only the inclusion map $A\hookleftarrow B$ if $B\subseteq A$ and is empty otherwise.
\item[(b)] The product of $A$ and $B$ is given by the union $A\cup B$, and the fact that all morphisms are epimorphisms implies this is an epi-product too.
\item[(c)] Any sequence of finite moments $M_A$ on this category is well-behaved in the sense of \cite{sawin-wood2022}, because the well-behavedness sum is always finite. A level is the power set $2^D$ for some finite set $D$, while the elements with epimorphisms from $A$ are precisely the subsets of $A$.
\item[(d)] The M\"obius function on the lattice of isomorphism classes ordered by epimorphisms is given by $\mu(B,A) = (-1)^{|A\setminus B|}$.
\end{itemize}

The function $\#{\rm Epi}(B,A)$ is then the characteristic function of the event $(A\subseteq B)$, which among a level $2^D$ has expected value precisely $\prod_{a\in A} r_a$, so we define this to be $M_A$. Certainly $M_{A\cup B} = M_A M_B$ whenever $A\cap B = \emptyset$. We check that $M_A$ corresponds to a measure on the pro-objects using \cite[Theorem 1.7]{sawin-wood2022}. Indeed,
\begin{align*}
v_{2^D,B} &= \sum_{A\subseteq D} \frac{\hat{\mu}(B,A)}{|\Aut(A)|}M_A\\
&=\sum_{B\subset A\subseteq D} \frac{(-1)^{|A\setminus B|}}{1}\prod_{a\in A} r_a\\
&=\prod_{b\in B} r_b \prod_{d\in D\setminus b} (1-r_d)\\
&\ge 0
\end{align*}
so the measure $\mu$ exists.

We also consider that $\#{\rm Epi}(\mathscr{G},A)$ is the characteristic function of the event $(A\subseteq \mathscr{G})$. Thus, we conclude that
\[
\#{\rm Epi}(\mathscr{G},A)\#{\rm Epi}(\mathscr{G},B) = \#{\rm Epi}(\mathscr{G},A\cup B),
\]
so that $M_{(A_1,A_2,...,A_{2k})} = M_{A_1\cup A_2\cup \cdots \cup A_{2k}}$. In this example, we have a very convenient bound for the finite moments of unions
\[
M_{A\cup B} = \prod_{a\in A\cup B} r_a \ge \prod_{a\in A} r_a \prod_{b\in B} r_b = M_A M_B
\]
which follows from the assumption that $r_n\in [0,1]$. In particular, this implies $M^{(j)} \le M^{(2k)}$ for each $j\le 2k$.

Any ordering $f_n$ supported on a family of pairwise disjoint sets necessarily satisfies that $C^{2k}\setminus E(2k,M)$ intersected with the support of $f_n$ is precisely the collection of tuples $(A_1,...,A_{2k})$ in the support of $f_n$ which has at most $k$ distinct coordinates. We choose $f_n$ to be the characteristic function of singleton sets $\{m\}$ for which $m\le n$. Up to the number of ways to choose matching coordinates, it suffices to consider objects in $C^k$. Thus, we can bound
\begin{align*}
\int_{C^{2k} \setminus E(2k,M)} |f_n|\ dM^{(j)}(dM)^{2k-j} &\le \int_{C^{2k} \setminus E(2k,M)} |f_n|\ dM^{(2k)}\\
&\ll_k \sum_{(A_1,...,A_k)\in C^k} f_n(A_1)\cdots f_n(A_k)M_{A_1\cup\cdots \cup A_k}
\end{align*}
Noting that $f_n$ is supported on singleton sets, this is equivalent to
\begin{align*}
&=\sum_{|A|\le k} \#\{a_1,...,a_k\in A : A = \{a_1,...,a_k\}\} \cdot\prod_{a\in A} f_n(\{a\})M_{\{a\}}\\
&\ll_k \sum_{1\le |A|\le k}\prod_{a\in A} f_n(\{a\})M_{\{a\}}\\
&=\left(1 + \int_C f_n\ dM\right)^k - 1.
\end{align*}
If $\int_C f_n\ dM\ge 1$, we can bound this by a constant times the leading term $(\int_C f_n\ dM)^k$, and otherwise this is $O(1)$. Thus we conclude via Theorem \ref{thm:bounding_mixed_moments_2k} that
\begin{align*}
\int_{\mathcal{P}}\left\lvert N(\mathscr{G},f_n) - \int_C f_n\ dM\right\rvert^{2k} d\mu(\mathscr{G}) &\ll_k \left(\int_C f_n\ dM\right)^k + O(1).
\end{align*}
Taking $E(n) = \left(\int_C f_n\ dM\right)^{1/2}$ and $\gamma(n) = n^2$ in Theorem \ref{thm:SLLN} then implies
\[
\frac{N(\mathscr{G},f_n) - \displaystyle\int_C f_n\ dM}{n^{1/k} \sqrt{\displaystyle\int_C f_n\ dM}} \overset{a.s.}{\longrightarrow} 0.
\]

The counting function is given by
\begin{align*}
N(\mathscr{G},f_n) &= \sum_{m\le n}\#{\rm Epi}(\mathscr{G},\{m\})\\
&= \#(\mathscr{G}\cap \{1,...,n\}),
\end{align*}
while the moments of the ordering are given by
\begin{align*}
\int_C f_n\ dM &=\sum f_n(A) M_A\\
&= \sum_{m\le n} r_m.
\end{align*}
Taking $k$ sufficiently large concludes the proof.
\end{proof}

Keeping in mind the connection with random models for prime numbers, we also prove the following:
\begin{corollary}
Let $\mathscr{G}$ be a random subset of $\N$ where we let the events $(n\in \mathscr{G})$ be pairwise independent with probability $r_n\in [0,1]$. Let $\mathbf{1}_S$ be a characteristic function of some subset $S\subseteq \N$, and define the corresponding ordering $f_n$ supported on singleton sets by $f_n = \mathbf{1}_{S\cap \{1,...,n\}}$. Then
\[
\frac{\displaystyle \#\{m\le n : m\in \mathscr{G} \cap S\} - \sum_{\substack{j\le n\\j\in S}}r_j}{\displaystyle n^{\epsilon}\sqrt{\sum_{\substack{j\le n\\j\in S}} r_j} } \overset{a.s.}{\longrightarrow} 0
\]
as $n\to \infty$.
\end{corollary}

If we let $r_n$ be the probabilities of Cram\'er's model (or of the modifications improving the model), Heuristic \ref{heur:vast} makes predictions for sets of prime numbers belonging to any subset $S\subseteq \N$ obeying the incredibly mild density condition
\[
\sum_{\substack{j\le n\\j\in S}}r_j \gg n^{\delta}
\]
for some $\delta > 0$. This prediction is extremely broad, where the currently ``known issues" amount to divisibility by small primes which are accounted for in corrections to Cram\'er's random model (see \cite{granville1995} for a nice summary). This example is not ``novel'', such measure $1$ statements exist throughout the literature on random models for the primes generalizing Cram\'er's work as a means to justify a number of the most famous conjectures on the distribution of primes. This includes the likes of the Hardy-Littlewood conjecture (where $f$ is the characteristic function on admissible constellations starting from $j$) and primes of the form $x^2+1$ (where $f$ is the characteristic function of integers of the form $x^2+1$).

In addition to being a short example with nice properties, this is meant to demonstrate two things. Firstly this example demonstrates that important existing random models (and corresponding results for those models) for the distribution of prime numbers arise as special cases of Heuristic \ref{heur:vast} (respectively Theorem \ref{thm:LLNintro}) for the category of subsets of $\N$. Secondly, this example demonstrates the power of working at this level of generality. We have given a concrete description for when such probability $1$ results will exist, which allow us to tackle numerous orderings at a time.

\subsection{Random groups}

Any sequence $M_G = O(|G|^n)$ on the category of finite groups is well-behaved in the sense of \cite{sawin-wood2022}, so their results can be used to determine when such a moment sequence gives rise to a probability distribution. We will focus on the example $M_G=1$ discussed in \cite{sawin-wood2022}. We can use Theorem \ref{thm:LLNintro} to make predictions for (asymptotically) how many quotients of a ``random" profinite group have a given behavior. While questions of this nature are easy enough to formulate, the moments of the most natural orderings $f_n$ can be much harder to compute in practice. Utilizing the results of the preceeding subsection, we give the following example:

\begin{theorem}
Let $M_G=1$ on the category of finite \emph{abelian} groups with associated probability measure $\mu$ on the profinite abelian groups. Then the average number of maximal subgroups in a random pro-abelian group of index bounded above by $n$ tends to $\log\log n$ in almost surely. More specifically,
\[
\frac{\#\{N\le \mathscr{G} \text{ maximal}: [\mathscr{G}:N]\le n\}}{\log\log n} \overset{a.s.}{\longrightarrow} 1
\]
as $n\to \infty$.
\end{theorem}

\begin{proof}
We take the ordering
\[
f_n(G) = \begin{cases}
\frac{1}{p-1} & G \cong C_p,\ p\le n\\
0&\text{else}.
\end{cases}
\]
All maximal subgroups of an abelian group are normal with quotient isomorphic to $C_p$ for some prime $p$, so we can compute
\begin{align*}
\#\{N\le \mathscr{G} \text{ maximal}: [\mathscr{G}:N]\le n\} &= \sum_{p\le n} \frac{\#{\rm Epi}(\mathscr{G},C_p)}{p-1}\\
&=\sum_{G\in C/\cong} f_n(G)\#{\rm Epi}(\mathscr{G},C_p)\\
&=N(\mathscr{G},f_n).
\end{align*}
We note that $f_n$ is nonegative and that $n\mapsto N(\mathscr{G},f_n)$ is certainly increasing in $n$. Thus, we can apply Theorem \ref{thm:LLNintro}(iii). We first compute
\begin{align*}
\int_C f_n\ dM &= \sum_{p\le n} \frac{1}{p-1} \sim \log\log n.
\end{align*}
We next apply Theorem \ref{thm:bounding_mixed_moments_2} to bound (\ref{eq:main_bound}). The ordering $f_n$ is supported on the finite simple abelian groups $C_p$, which are of pairwise coprime order. Thus, the intersection of $C^2\setminus E(2,M)$ with the support of $f_n$ is precisely the diagonal objects $(C_p,C_p)$ with $p\le n$. We compute
\begin{align*}
\int_{C^2\setminus {\rm Epi}_M^2 C^2} f_n\ dM^{(2)} &= \sum_{p\le n} \frac{1}{(p-1)^2} M_{(C_p,C_p)}\\
&= \sum_{p\le n} \frac{1}{(p-1)^2} \left(M_{C_p\times C_p} - (p-1) M_{C_p}\right)\\
&= \sum_{p\le n} \frac{p}{(p-1)^2}\\
&\sim \log\log n,
\end{align*}
where the second equality follows from Lemma \ref{lem:MG1Gk}, noting that the subgroups of $C_p\times C_p$ that surject onto each coordinate are the whole group, and $p-1$ subgroups isomorphic to $C_p$. The $j=1$ integral actually converges as $n\to \infty$ by a similar computation, so this is the maximum. All together Theorem \ref{thm:bounding_mixed_moments_2} implies
\begin{align*}
\int_{\mathcal{P}}\left\lvert N(\mathscr{G},f_n) - \int_C f_n\ dM\right\rvert^2 d\mu(\mathscr{G}) &= O(\log\log n)\\
&= O\left(\frac{(\log\log n)^2}{(\log\log n)^{1/2}}\right)\\
&= O\left(\frac{\left(\int_C f_n dM\right)^2}{\psi\left(\int_C f_n dM\right)}\right),
\end{align*}
where $\psi(t) = t^{1/2}$ is non decreasing with $\lim_{t\to \infty} \psi(t) = \infty$. We confirm that $\sum \frac{1}{n\psi(n)} = \sum n^{-3/2} < \infty$ converges, so the result follows from Theorem \ref{thm:LLNintro}(iii).
\end{proof}

More general questions on random groups can be asked, but much about the asymptotic behavior of large families of groups is difficult to compute. This makes asymptotic growth rates and error terms also difficult to compute in practice. For instance, it would be nice to study the ordering $f_n(G) = \frac{1}{|\Aut(G)|}$ if $|G|\le n$ and $0$ otherwise. The corresponding counting function is then
\[
N(\mathscr{G},f_n) = \#\{N\normal \mathscr{G} \mid [\mathscr{G}:N] \le n\}.
\]
This is a naturally interesting function to ask about. However, the corresponding moment
\[
\int_C f_n\ dM = \sum_{|G|\le n} \frac{1}{|\Aut(G)|}
\]
is not easy to determine. For instance, there are reasons to suspect that 100\% of groups ordered by cardinality are $2$-groups, but the author is not aware of a proof of this statement. Moreover, the bound in Theorem \ref{thm:bounding_mixed_moments_2} requires some more serious group theory to translate the results of Lemma \ref{lem:MG1Gk} into bounds on the error.

A number of interesting statistical questions that are simple to state arise in this way. The language of Theorem \ref{thm:LLNintro} is useful to create a framework for determining asymptotic growth rates, even if the moments of the ordering requires additional study to prove any results.

\bibliographystyle{alpha}
%\bibliography{../../BAreferencesV2020}
\bibliography{Counting_functions.bbl}

\begin{thebibliography}{LWZB19}

\bibitem[Alb23]{alberts2023}
Brandon Alberts.
\newblock A random group with local data, 2023.
\newblock Forthcoming.

\bibitem[BBH17]{boston-bush-hajir2017}
Nigel Boston, Michael~R. Bush, and Farshid Hajir.
\newblock Heuristics for $p$-class towers of imaginary quadratic fields.
\newblock {\em Mathematische Annalen}, 368(1-2):633--669, June 2017.

\bibitem[BM90]{batyrev-manin1990}
V.~V. Batyrev and Yu.~I. Manin.
\newblock Sur le nombre des points rationnels de hauteur born{\'e} des
  vari{\'e}t{\'e}s alg{\'e}briques.
\newblock {\em Mathematische Annalen}, 286(1-3):27–43, 1990.

\bibitem[CL84]{cohen-lenstra1984}
Henri Cohen and Hendrik.~W. Lenstra.
\newblock Heuristics on class groups of number fields.
\newblock {\em Lecture Notes in Mathematics Number Theory Noordwijkerhout
  1983}, pages 33--62, 1984.

\bibitem[Cra94]{cramer1994}
Harald Cram{\'e}r.
\newblock Some theorems concerning prime numbers.
\newblock {\em Springer Collected Works in Mathematics}, page 138–170, 1994.

\bibitem[Ele18]{mathstackLLN2}
(https://math.stackexchange.com/users/7145/elements) Elements.
\newblock Weak law of large numbers for dependent random variables with bounded
  covariance.
\newblock Mathematics Stack Exchange, 2018.
\newblock URL:https://math.stackexchange.com/q/245327 (version: 2018-04-10).

\bibitem[Ete81]{etemadi1981}
N.~Etemadi.
\newblock An elementary proof of the strong law of large numbers.
\newblock {\em Z. Wahrscheinlichkeitstheorie und Verwandte Gebiete},
  55(1):119–122, 1981.

\bibitem[Ete83a]{etemadi1983a}
N.~Etemadi.
\newblock On the laws of large numbers for nonnegative random variables.
\newblock {\em Journal of Multivariate Analysis}, 13(1):187–193, 1983.

\bibitem[Ete83b]{etemadi1983b}
N.~Etemadi.
\newblock Stability of sums of weighted nonnegative random variables.
\newblock {\em Journal of Multivariate Analysis}, 13(2):361–365, 1983.

\bibitem[Fis11]{fischer2011}
Hans Fischer.
\newblock {\em A history of the central limit theorem: From classical to modern
  probability theory}.
\newblock Springer New York, 2011.

\bibitem[FMT89]{franke-manin-tschinkel1989}
Jens Franke, Yuri~I. Manin, and Yuri Tschinkel.
\newblock Rational points of bounded height on fano varieties.
\newblock {\em Inventiones Mathematicae}, 95(2):421–435, 1989.

\bibitem[FW89]{friedman-washington1989}
Eduardo Friedman and Lawrence~C. Washington.
\newblock On the distribution of divisor class groups of curves over a finite
  field.
\newblock {\em Th{\'e}orie des nombres / Number Theory}, page 227–239, 1989.

\bibitem[Gra95]{granville1995}
Andrew Granville.
\newblock Harald cram{\'e}r and the distribution of prime numbers.
\newblock {\em Scandinavian Actuarial Journal}, 1995(1):12–28, 1995.

\bibitem[KP10]{korchevsky-petrov2010}
V.~M. Korchevsky and V.~V. Petrov.
\newblock On the strong law of large numbers for sequences of dependent random
  variables.
\newblock {\em Vestnik St. Petersburg University: Mathematics},
  43(3):143–147, 2010.

\bibitem[LWZB19]{liu-wood-zureick-brown2019}
Yuan Liu, Melanie~Matchett Wood, and David Zureick-Brown.
\newblock A predicted distribution for {G}alois groups of maximal unramified
  extensions, July 2019.
\newblock Preprint available at \url{https://arxiv.org/abs/1907.05002}.

\bibitem[Mal02]{malle2002}
Gunter Malle.
\newblock On the distribution of {Galois} groups.
\newblock {\em Journal of Number Theory}, 92(2):315--329, 2002.

\bibitem[Mal04]{malle2004}
Gunter Malle.
\newblock On the distribution of {Galois} groups, {II}.
\newblock {\em Experimental Mathematics}, 13(2):129--135, 2004.

\bibitem[Mic17]{mathstackLLN1}
(https://math.stackexchange.com/users/155065/michael) Michael.
\newblock Pairwise uncorrelated random variables in strong law of large numbers
  (slln).
\newblock Mathematics Stack Exchange, 2017.
\newblock URL:https://math.stackexchange.com/q/2545239 (version: 2017-11-30).

\bibitem[MU17]{mitzenmacher-upfal2017}
Michael Mitzenmacher and Eli Upfal.
\newblock {\em Probability and computing: Randomization and probabilistic
  techniques in algorithms and data analysis}.
\newblock Cambridge University Press, 2017.

\bibitem[Pet09a]{petrov2009a}
V.~V. Petrov.
\newblock On stability of sums of nonnegative random variables.
\newblock {\em Journal of Mathematical Sciences}, 159(3):324–326, 2009.

\bibitem[Pet09b]{petrov2009b}
V.~V. Petrov.
\newblock On the strong law of large numbers for nonnegative random variables.
\newblock {\em Theory of Probability \&; Its Applications}, 53(2):346–349,
  2009.

\bibitem[Sen13]{seneta2013}
Eugene Seneta.
\newblock A tricentenary history of the law of large numbers.
\newblock {\em Bernoulli}, 19(4), 2013.

\bibitem[SS93]{sen-singer1993}
Pranab~Kumar Sen and Julio da~Motta Singer.
\newblock {\em Large sample methods in statistics: An introduction with
  applications}.
\newblock Chapman {\&} Hall/CRC, 1 edition, 1993.

\bibitem[SW22]{sawin-wood2022}
Will Sawin and Melanie~Matchett Wood.
\newblock The moment problem for random objects in a category, Oct 2022.

\bibitem[wik22]{wikipediaLLN}
Law of large numbers, Oct 2022.

\end{thebibliography}

\end{document}